\def\ArXivS{1}
\ifx\ArXivS\undefined
\documentclass[authoryear,hyperref,preprint,12p]{elsarticle}
\ifx\StProb\undefined
\makeatletter%
\def\ps@pprintTitle{%
     \let\@oddhead\@empty
     \let\@evenhead\@empty
     \def\@oddfoot{}%
     \let\@evenfoot\@oddfoot}
\makeatother%
\fi
\else
\documentclass[10p,a4paper]{article}
\newcommand{\keywords}[1]{\gdef\and{; }\begin{description}
\item[Keywords]{#1}
\end{description}}
\newcommand{\subclass}[1]{\gdef\and{; }\begin{description}
\item[Subclass]{#1}
\end{description}}
\newlength{\ei}\ei=0.0138888889mm
\newcommand{\qed}{\nopagebreak\hspace*{\fill}\mbox{\scriptsize\rm%
  {/}\hspace{-13\ei}{/}\hspace{-13\ei}{/}\hspace{-13\ei}{/}}}   
\fi

\newenvironment{proof}[1]
{\begin{trivlist}\item[]
 \ifthenelse{\equal{#1}{ }}%
   {{\sc Proof:}}%
   {{\sc Proof\/ #1\/:}}%
}
{\hfill  \qed\end{trivlist}}%

\usepackage{graphicx}

\usepackage{amsmath}
\usepackage{subfigure}
\usepackage{amssymb}
\usepackage{mathptmx}      

\RequirePackage{natbib}
\RequirePackage{hypernat}
\usepackage{url}
\newcommand{\href}[2]{{#2}}
\RequirePackage{ifthen}
\usepackage{color}
\definecolor{red}{RGB}{255,0,0}
\newcommand{\NewInRev}[1]{{{#1}}}
\newcommand{\Ss}{\scriptstyle}  \newcommand{\SSs}{\scriptscriptstyle}

\newcommand{\rfia}[1]{\makebox[\parindent][l]{%
                     \makebox[0em][r]{\rm(}\sf#1\rm)}}
\newcounter{ABCcB}
\newcommand{\theABCcC}{\alph{ABCcB}}

\newcommand{\Ew}{\mathop{\rm {{}E{}}}\nolimits} 

\newcommand{\Lo}{\mathop{\rm {{}o{}}}\nolimits}

\newcommand{\tr}{\mathop{\rm{} tr{}}}
\newcommand{\B}{\mathbb B}

\newcommand{\R}{\mathbb R}

\newcommand{\N}{\mathbb N}
\newcommand{\Jc}{\mathop{\bf\rm{{}I{}}}\nolimits}

\newcommand{\sR}{{\Ss \mathbb R}}

\newcommand{\Tfrac}[2]{\textstyle\frac{#1}{#2}}

\newlength{\SyW}   \newlength{\msu}  \msu=\mathsurround 

\newtheorem{Thm}{Theorem}
\newtheorem{Prop}[Thm]{Proposition}
\newtheorem{Lem}[Thm]{Lemma}
\newtheorem{Rem}[Thm]{Remark}

\newtheorem{Def}[Thm]{Definition}

\newtheorem{Exam}[Thm]{Example}

\numberwithin{equation}{section}
\numberwithin{Thm}{section}

\newcounter{ABCc}
\renewcommand{\theABCc}{\alph{ABCc}}

\newcommand{\Tpi}{{\SSs\;\textrm{T}_{\SSs\pi}}}
\newcommand{\dpi}{{\,{\cdot_{\SSs\pi}}\,}}
\newcommand{\Trsp}{{\SSs\;\textrm{T}}}

\begin{document}

\ifx\ArXivS\undefined%

\begin{frontmatter}
\author[hfu,tukl]{Daria Pupashenko}
{\author[itwm]{Peter Ruckdeschel\corref{cor1}}
\ead{peter.ruckdeschel@itwm.fraunhofer.de}
\cortext[cor1]{Corresponding author}}
\author[hfu]{Matthias Kohl}
\address[hfu]{Furtwangen University, Department of Medical and Life Sciences,
  			  Jakob-Kienzle-Str. 17, 78054 VS-Schwenningen, Germany}
\address[tukl]{TU Kaiserslautern, AG Statistik, FB.\ Mathematik,
         P.O.Box 3049, 67653 Kaiserslautern, Germany}
\address[itwm]{Fraunhofer ITWM, Abt.\ Finanzmathematik,
         Fraunhofer-Platz 1, 67663 Kaiserslautern, Germany}
\date{\today}

\title{$L_{2}$ Differentiability of Generalized Linear Models}


\begin{abstract}
We derive conditions for $L_2$ differentiability of generalized linear models
with error distributions not necessarily belonging to exponential families,
covering both cases of stochastic and deterministic regressors. These conditions
induce smoothness and integrability conditions for corresponding GLM-based
time series models.
\end{abstract}
\end{frontmatter}
%
%


\else
\title{$L_{2}$ Differentiability of Generalized Linear Models}
\author{Daria Pupashenko\footnotemark\;\,\footnotemark\;,
        Peter Ruckdeschel\footnotemark\;, and
        Matthias Kohl${}^{\ast}$}
\date{\today}
\maketitle%
\renewcommand{\thefootnote}{\fnsymbol{footnote}}
\footnotetext[1]{Furtwangen University, Dept.\ of Medical and Life Sciences,
							Jakob-Kienzle-Str. 17, 78054 VS-Schwenningen, Germany}%
\footnotetext[2]{Dept.\ of Mathematics, University of Kaiserslautern,
              P.O.Box 3049, 67653 Kaiserslautern, Germany}%
\footnotetext[3]{Fraunhofer ITWM, Dept.\ of Financial Mathematics,
              Fraunhofer-Platz 1, 67663 Kaiserslautern, Germany\\
              corresponding address: \texttt{peter.ruckdeschel@itwm.fraunhofer.de}}%
%
\begin{abstract}
We derive conditions for $L_2$ differentiability of generalized linear models
with error distributions not necessarily belonging to exponential families,
covering both cases of stochastic and deterministic regressors. These conditions
induce smoothness and integrability conditions for corresponding GLM-based
time series models.
\subclass{MSC 62F12, 62F35}
\end{abstract}
\fi

\section{Motivation}\label{setup}
%
%
Introduced by \citet{N:W:72}, generalized linear models (GLMs) have become one of
the most frequently used statistical models with a vast amount of published
results. Hence, trying to give a full account on relevant literature would be
pretentious. We instead refer to the monographs \citet{MC:Ne:89} and \citet{Fa:Tu:01}.
When it comes to regularity assumptions, though, this literature focuses on
GLMs which are exponential families, compare \citet{Hab:74,Hab:77,Fah:90,Fa:Ka:85},
or uses quasi-likelihood or pseudo-likelihood techniques to account for
over/under-dispersion effects, see\ifx\ArXivS\undefined%
, a.o.,
\fi\citet{G:M:T:84,Ne:Pr:87,MC:Ne:89}.
In some situations, exponential families are a too narrow class, though: E.g., recently
log-linear models for generalized Pareto distributions have found applications in
operational risk (compare \citet{D:G:10}), but distributions of extreme value type with unknown
shape parameter do not fall into the range of exponential families and so far are not
yet covered.

Heading for asymptotic results and robustness, we are not only interested in consistency
results for specific estimators like maximum likelihood estimators (MLEs), but rather in
\textit{local asymptotic normality (LAN)} in the sense of \citet{LC:70,Haj:72}.
\NewInRev{With the LAN property at hand a very powerful asymptotic framework as pioneered by Le\!~Cam
is available: It gives a precise setup in which to obtain strong optimality results
for (estimators behaving asymptotically like) the MLE, i.e., the Asymptotic Convolution
Theorem and the Asymptotic Minimax
Theorem, see, e.g.\ \citet[Thms.~3.2.3 \& 3.3.8]{Ri:1994} or \citet[Thms.~8.8 \& 8.11]{vdV:1998}.
The LAN property entails necessary expansions for asymptotic maximin tests with explicit terms for the
asymptotic maximin power under local alternatives \citep[Sec.11.9]{LC:86}; it is the
starting point for efficient and adaptive estimation in semiparametric models (compare
\citet{BKRW:1993}) and for a comprehensive theory of optimally-robust procedures
(see \citet[Chs.~5 \& 7]{Ri:1994}).}

\NewInRev{Now, a sufficient condition for the LAN property is given by $L_2$-differentiability
(see, e.g.\ \citet[Thm.~2.3.5]{Ri:1994}), and---at least in the i.i.d.\ setting---this
is a necessary condition, too, compare \citet[Ch.~7, Prop.~3]{LC:Y:00}. Hence in this light,}
deriving smoothness of the model in terms of $L_2$-differentiability would be a desirable
goal; i.e., to consider GLMs as particular parametric models and to derive their
$L_2$-differentiability.
For GLMs which are exponential families, this has already been achieved in \citet{Schl:1994}.
Typically, however, scale-shape families as e.g.\ the generalized Pareto distributions
are non-exponential. In this article, we hence generalize results of \citet[Sec.~2.4]{Ri:1994}
on $L_2$-differentiability for linear regression models to also cover error distributions
with a $k$-dimensional parameter and with regressors of possibly different length for each parameter.
More specifically, we separately treat the case of stochastic regressors, which is of particular
interest for incorporating (space-)time dependence, and of deterministic regressors as occurring
in planned experiments.

While in principle $L_2$-differentiability of these models could be settled by
general auxiliary results from \citet[Lem.~A.1--A.3]{Haj:72}, or be placed
in the framework of \citet{Ri:Ru:2001}, our goal are sufficient conditions
directly exploiting the regression structure. More specifically, these conditions refer to
(i) smoothness of the error distribution model, (ii) (uniform) integrability of the scores
($L_2$-derivative) and (iii) suitably integrated continuity of the Fisher information of
again the error distribution model.

At first glance, this might look like a technical exercise but setting up time series models
where time-dependence is captured by a GLM-type link with (functions of) the own past observations
as regressors,  conditions~(ii) and~(iii) reveal to which extent the current error distribution
may depend upon the past without making it ``over-informative'' for the present.
\NewInRev{More precisely, letting aside dimensionalities of the parameter of the error distribution
and the regressors, the scores function of a GLM ${\cal P}$ with errors from a distribution model
${\cal Q}$, link function $\ell$ and regressor $x$ is of form
$\Lambda^{\cal P}_\beta(x,y)=\Lambda_{\ell(x\beta)}^{\cal Q}(y) \dot\ell(x\beta) x$, where
$\Lambda^{\cal Q}_\vartheta$ are the parametric scores from model ${\cal Q}$. Now even
if ${\cal Q}$ has fat tails and non-existing moments, in many cases $\Lambda^{\cal Q}_\vartheta$
still is square integrable, see e.g.\ the case of $\alpha$-stable distributions as in \citet{DuMouchel1973}
or the generalized extreme value and Pareto distributions GEVD and GPD
explicated later on in this paper. If however, as in
a autoregressive (AR) time series context with identity link $\ell(\theta)=\theta$, $x$
comes again from a distribution within ${\cal Q}$, the LAN property may fail due to
a lack of integrability. This is the case in \citet[Thm.~3.3]{ACD:2009}, where in addition
the authors obtain slower convergence rates for $\beta$ in an AR-model with $\alpha$-stable errors.
One way to preserve the LAN property could consist in using a suitable link function
$\ell$ such that the product $\dot \ell x$ becomes square integrable---see later
in this paper for corresponding GPD and GEVD time series. This technique can be seen as
an alternative / an extension to the approach using regression ranks as in \citet{HSVV:2011},
which in the respective case of a regression model with $\alpha$-stable errors and deterministic
regressors achieves the same goal, i.e.,  extending the availabilty of the LAN property.
}

\NewInRev{In this paper, we explicate } the respective conditions (i)--(iii) for the cases of
stochastic and deterministic regressors,
respectively, in examples including---for reference and com\-parison---linear regression,
Poisson, and Binomial regression, as well as scale-shape regression for the GPD and GEVD.

In particular for the latter distributions we give conditions which render a corresponding
time series model accessible to the LAN type framework and thus contribute a new sort of GLM for
extreme value type distributions where the tail weight respectively, the shape parameter depends
on past observations in an autoregressive way. Thus, large extreme observations may foster or
dampen the occurrence of future large extreme observations and controlling the extremal index
(see \citet[p.413--423]{E:C:M:97}) this way.

The rest of the paper is organized as follows: Section~\ref{Mainsec} provides the mathematical
setup and the main results with Theorem~\ref{Thmmain} (for random carriers) and Theorem~\ref{ThmmainDet}
(for deterministic carriers). The examples are worked out in Section~\ref{examples}.
The proofs of our assertions are given in the appendix.

\section{Main Results}\label{Mainsec}
%
Let $(\Omega, {\cal A})$ be a measurable space and ${\cal M}_{1}({\cal A})$ the set of all
probability measures on~${\cal A}$. Consider
$\mathcal{Q}=\{Q_{\vartheta}|\vartheta\in \Theta\}\subset {\cal M}_{1}({\cal A})$ a parametric
model with open parameter domain $\Theta\subset \R^{k}$. Following Le\!~Cam and Rieder, we write
$dQ_{\vartheta}$ for the densities w.r.t.\ some dominating measure $\nu$ on ${\cal A}$ and denote
the norm in the respective $L_2(\nu)$ space by $\|\,\cdot\|_{\mathcal{L}_2}$; as usual, $\nu$
is suppressed from notation as the choice of $\nu$ has no effect on respective convergence assertions.
%
In this context,  $L_{2}$ differentiability in the case of i.i.d.\ observations is defined as follows.
\begin{Def}
\label{def}
Model $\mathcal{Q}$ is called $L_{2}$ differentiable at $\vartheta\in\Theta$ if there exists a function
$\Lambda^{\SSs\mathcal{Q}}_{\vartheta} \in L_{2}^{k}(P_{\vartheta})$ such that, as $h\rightarrow 0 \in\R^k$
\begin{equation} \label{L2dk}
\left\|\sqrt{dQ_{\vartheta+h}}-\sqrt{dQ_{\vartheta}}\big(1+\tfrac{1}{2}(\Lambda^{\SSs\mathcal{Q}}_{\vartheta})^\Trsp h\big)\right\|_{\mathcal{L}_{2}}=\Lo(|h|).
\end{equation}
Then, $\Lambda^{\SSs\mathcal{Q}}_{\vartheta}$ is the $L_{2}$ derivative and the $k\times k$ matrix
$\mathcal{I}^{\SSs\mathcal{Q}}_{\vartheta}=\Ew_{\vartheta}\Lambda^{\SSs\mathcal{Q}}_{\vartheta}(\Lambda^{\SSs\mathcal{Q}}_{\vartheta})^\Trsp $
is the Fisher information of $\mathcal{Q}$ at $\vartheta$.\smallskip\\
We say that $\mathcal{Q}$ is continuously $L_{2}$ differentiable at $\vartheta$ if, for any $h\to 0 \in\R^k$,
\begin{equation} \label{L2dc}
\sup_{t\in\sR^k\colon\,|t|\leq 1}\left\|\sqrt{dQ_{\vartheta+h}} (\Lambda^{\SSs\mathcal{Q}}_{\vartheta+h})^\Trsp t -
\sqrt{dQ_{\vartheta}} (\Lambda^{\SSs\mathcal{Q}}_{\vartheta})^\Trsp t\right\|_{\mathcal{L}_{2}}=\Lo(1).
\end{equation}
\end{Def}


Introducing regressors to explain parameter $\vartheta$, we turn model $\mathcal{Q}$ into a
regression model $\mathcal{P}$ with parameter $\beta$.
To this end, for $p\in \N$, let $\pi\in\N^k$, $\pi=(p_h)_{h=1,\ldots,k}$ be a partition of the $p$ coordinates
into blocks of dimension $p_h$, i.e., $\sum_h p_h = p$. Obviously, then each $x\in\R^p$ can unambiguously be indexed
by the double index $(x_{h,j})_{{{h=1,\ldots,k}\atop{j=1,\ldots,p_h}}}$. For these blocks we define the following
operators:
\begin{eqnarray}
&&T_\pi\colon \R^p\times\R^p \to \R^k,\quad (a,b)\mapsto T_\pi(a,b)=:a^{\Tpi}b=(\sum_{j=1}^{p_h} a_{h,j} b_{h,j})_{h=1,\ldots,k}\\
&&\rho_\pi\colon \R^k\times\R^p \to \R^p,\quad (c,a)\mapsto \rho_\pi(c,a)=:c\,\dpi\,a=(c_h a_{h,j})_{{{h=1,\ldots,k}\atop{j=1,\ldots,p_h}}}\\
&&M_\pi\colon \R^{k\times k}\times\R^p \times\R^p \to \R^{p\times p},\quad (C,a,b)\mapsto M_\pi(C,a,b)=\nonumber\\
&&\mbox{\hspace{6cm}}(C_{h_1,h_2} a_{h_1,j_1} b_{h_2,j_2})_{{{h_1,h_2=1,\ldots,k}\atop{j_1,j_2=1,\ldots,p_h}}}.
\end{eqnarray}
We also write $C\dpi a$ for a $k\times m$ matrix $C$, meaning that we
apply $\rho_\pi$ to $C$ column by column as first argument, so that the result will be
the respective $p\times m$ matrix $(c_{h,l} a_{h,j})_{{{h=1,\ldots,k}\atop{j=1,\ldots,p_h}};l=1,\ldots,m}$.
\medskip\\
Then, the case of a $k$-dimensional parameter $\vartheta$ in Model~$\mathcal{Q}$ and non-identically dimensional
regressors for each of the $k$ coordinates can be captured using a continuously differentiable link
function $\ell\colon\R^k\to\Theta$ with derivative $\dot\ell$, so that for a $p$-dimensional regressor $X$
and $p$-dimensional regression parameter $\beta$ we obtain a regression as $\vartheta=\ell(\theta)$ for
$\theta = X^{\Tpi} \beta$. Applying the chain rule, the candidate $L_2$ derivative in this regression model
is
\begin{equation}\label{LambdaPdef}
\Lambda_\beta^{\SSs\mathcal{P}}(x,y)=\dot \ell(\theta)^\Trsp  \Lambda_\vartheta^{\SSs\mathcal{Q}}(y) \,\dpi\,x\,.
\end{equation}
The case of the linear regression model treated in \citet[Sec.~2.4]{Ri:1994} is obtained as a special case
for $\mathcal{Q}$ an $L_2$-differentiable $k=1$-dimensional location model and $\ell$ the identity.
As in \citet[Sec.~2.4]{Ri:1994}, we distinguish the cases of stochastic and deterministic regressors.

\smallskip
To apply conditions as in \citet{Haj:72}, we need the notion of \textit{absolute continuity} in
$k$ dimensions: Let $f\colon\R^k\to\R$; we call $f$ \textit{absolutely continuous}, if for all
$a,b\in\R^k$  the function $G\colon[0,1]\to\R$,
$s\mapsto G(s)=f(a+s(b-a))$ is \textit{absolutely continuous} (as usual, see \citet[chap.~6]{Ru:1986})%
\label{ACmult}.\medskip

For later reference we recall the results of \citet[Lem.~A.1--A.3]{Haj:72}:
\begin{Prop}[H\'ajek]\label{Hajekprop}
Assume that in some ${\vartheta_0}\in\Theta$ surrounded by
some open neighborhood $U$, model $\mathcal Q$ satisfies
\begin{enumerate}
\item[(H.1)] The densities $dQ_\vartheta(y)$ are absolutely
continuous in each $\vartheta\in U$  for $Q_{\vartheta_0}$-a.e.\ $y$.
\item[(H.2)] The derivative $\frac{\partial}{\partial \vartheta} dQ_\vartheta(y) =
\Lambda_\vartheta(y)\,dQ_\vartheta(y)$ exists in each $\vartheta\in U$
for $Q_{\vartheta_0}$-a.e.\ $y$.
\item[(H.3)] The Fisher information $\mathcal{I}_{\vartheta}=\int\, \Lambda_{\vartheta}(y) \Lambda_{\vartheta}(y)^\Trsp\,Q_\vartheta(dy)$ exists, (i.e., the integral is finite) and is continuous in $\vartheta$ on $U$.
\end{enumerate}
Then, $\mathcal Q$ is continuously $L_2$ differentiable in $\vartheta_0$ with derivative
$\Lambda_{\vartheta_0}$ and Fisher information $\mathcal{I}_{\vartheta_0}$.
\end{Prop}


\subsection{Random Carriers}\label{RC}

In this context the regressors $x$ are stochastic with distribution $K$,
but the observations $(x,y)_i$ are then modeled as i.i.d.\ observations.
To this end, let model $\mathcal{Q}$ be a $k$-dimensional $L_2$-differentiable model with
parameter $\vartheta \in \Theta$ and derivative $\Lambda_\vartheta^{\SSs\mathcal{Q}}$ and
Fisher information $\mathcal{I}_\vartheta^{\SSs\mathcal{Q}}$. The corresponding GLM induced
by the link function $\ell\colon\R^k\to\Theta$ (with derivative $\dot\ell$) and partition $\pi$
is given as
\begin{equation} \label{PstoDef}
\mathcal{P} = \Big\{\, P_{\beta}(dx,dy)=Q_{\ell(x^{\Tpi} \beta)}(dy|x)\,K(dx) \;\mid\; \beta\in\R^p;\;Q_\vartheta\in\mathcal{Q}\,\Big \}\,.
\end{equation}

We state the following result.

\begin{Thm}\label{Thmmain}
Let $\beta_0\in\R^p$ and $\vartheta_t=\ell(\theta_t)$ for $\theta_t=x^{\Tpi}(\beta_0+t)$
as well as $\dot\ell_t=\dot\ell(\theta_t)$; further define
$\mathcal{I}^{\SSs\mathcal{P}}_{\vartheta_t}(x):= M_\pi\Big(\dot\ell_t^\Trsp  \mathcal{I}^{\SSs\mathcal{Q}}_{\vartheta_t} \dot\ell_t,x,x\Big)$.\smallskip\\
Then model $\mathcal{P}$ from \eqref{PstoDef} is $L_{2}$ differentiable in $\beta_0$
if subsequent conditions~(i)--(iii) hold.
\renewcommand{\labelenumi}{(\roman{enumi})}
\begin{enumerate}
\item Model $\mathcal{Q}$ fulfills (H.1)--(H.3) with ``$Q_{\vartheta_0}$-a.e. $y$'' replaced by ``$P_{\beta_0}$-a.e. $(x,y)$''
in (H.1) and (H.2).
\item
\begin{equation} \label{FiFinite}
\vspace{-1ex}\int |\mathcal{I}^{\SSs\mathcal{P}}_{\vartheta_0}(x)|\,K(dx) < \infty,
\end{equation}
\item for every $b\in(0,\infty)$,
\begin{equation} \label{FIcont1}
\vspace{-1ex}\lim_{s\to0}\sup_{|t|\leq b}\int \,\Big|\,|\mathcal{I}^{\SSs\mathcal{P}}_{\vartheta_{st}}(x)\,\,|- |\mathcal{I}^{\SSs\mathcal{P}}_{\vartheta_0}(x)|\,\, \Big|\,K(dx) = 0,
\end{equation}
where $|\mathcal{I}|$ is the Frobenius matrix norm, i.e., $|\mathcal{I}|^2=\tr \mathcal{I}^2$.
\end{enumerate}
Then model $\mathcal{P}$  is continuously $L_{2}$ differentiable in $\beta_0$ with derivative $\Lambda^{\SSs\mathcal{P}}_{\beta_0}(x,y)=
\dot\ell_0^\Trsp \Lambda^{\SSs\mathcal{Q}}_{\vartheta_0}(y) \,\dpi\,x$ and
Fisher information
$$\mathcal{I}^{\SSs\mathcal{P}}_{\beta_0}=\Ew_{\beta_0} \Lambda^{\SSs\mathcal{P}}_{\beta_0} (\Lambda^{\SSs\mathcal{P}}_{\beta_0})^\Trsp
= \int\,\mathcal{I}^{\SSs\mathcal{P}}_{\vartheta_0}(x) K(dx)\,.
$$
\end{Thm}
\begin{Rem} \label{Rem26}\rm\small
Sufficient conditions for \eqref{FiFinite} and \eqref{FIcont1} are
$
\vspace{-1ex}\int |\mathcal{I}^{\SSs\mathcal{P}}_{\vartheta_0}| \,|\dot{l}_0|^{2}\,|x|^{2}\,K(dx) < \infty,
$ 
and for every $b\in(0,\infty)$,
$
\vspace{-1ex}\lim_{s\to0}\sup_{|t|\leq b}\int \,\Big |\,\,|\mathcal{I}^{\SSs\mathcal{Q}}_{\vartheta_{st}}| \,|\dot{\ell}_{st}|^2-|\mathcal{I}^{\SSs\mathcal{Q}}_{\vartheta_0}|\, |\dot{\ell}_0|^{2} \,\,\Big| \, |x|^{2}\,K(dx) = 0
$. 

\end{Rem}


%
As just seen, the general GLM case comes with additional conditions for the link function~$\ell$ and its derivative.
For the linear regression case, they boil down to (i) $L_2$ differentiability
of the one dimensional location case and (ii) finite second moment of $x$ w.r.t. $K$.
(iii) becomes void, as $\dot \ell\equiv 1$ and $\mathcal{I}^{\SSs\mathcal{Q}}$ does not depend
on the parameter---compare \citet[Thm.~2.4.7]{Ri:1994}.

\subsection{Deterministic Carriers}\label{DC}

The case of deterministic carriers canonically leads to triangular schemes of independent, but no longer
identically distributed observations. To this end, we take up \citet[Def~2.3.8]{Ri:1994} and define
a corresponding notion of $L_2$-differentiability:


For $n\in\N$ and $i=1,\ldots,i_n$, let $(\Omega_{n,i}, {\cal A}_{n,i})$ be general sample spaces and
${\cal M}_{1}({\cal A}_{n,i})$ the set of all probability measures on ${\cal A}_{n,i}$. Consider the
array of parametric families of probability measures
$\mathcal{P}_{n,i}=\{P_{n,i,\beta}|\beta\in\R^p\}\subset {\cal M}_{1}({\cal A}_{n,i})$.

\begin{Def}
\label{def1}
The parametric array $\mathcal{P}=(\bigotimes_{i=1}^{i_n}\mathcal{P}_{n,i})$ is called $L_{2}$
differentiable at ${\beta_0}\in\R^p$ if there exists an array of functions
$\Lambda^{\SSs\mathcal{P}}_{n,i,{\beta_0}} \in L_{2}^{k}(P_{n,i,{\beta_0}})$ such that for all $i=1,...,i_{n}$
and $n\geq1$ the following conditions~\eqref{cent}--\eqref{cond2} are fulfilled.
\begin{equation}
\label{cent}
\Ew_{n,i,{\beta_0}}\Lambda^{\SSs\mathcal{P}}_{n,i,{\beta_0}}=0\,.
\end{equation}
Let $\mathcal{I}^{\SSs\mathcal{P}}_{n,i,{\beta_0}} = \Ew_{n,i,{\beta_0}}\Lambda^{\SSs\mathcal{P}}_{n,i,{\beta_0}}(\Lambda^{\SSs\mathcal{P}}_{n,i,{\beta_0}})^\Trsp $
and $\mathcal{I}^{\SSs\mathcal{P}}_{n,{\beta_0}}=\sum_{i=1}^{i_n}\mathcal{I}^{\SSs\mathcal{P}}_{n,i,{\beta_0}}$ and for $t\in\R^{k}$,
we define $t_{n}=(\mathcal{I}^{\SSs\mathcal{P}}_{n,{\beta_0}})^{-\frac{1}{2}}t$ and $U_{n,i}=U_{n,i,{\beta_0}}(t)=t_n^\Trsp \Lambda^{\SSs\mathcal{P}}_{n,i,{\beta_0}}$.
Then, for all $\epsilon\in(0,\infty)$ and all $t\in\R^{k}$ we require
\begin{equation}
\label{cond1}
\lim_{n\rightarrow\infty}\sum_{i=1,...,i_{n}}\int_{\left\{|U_{n,i}|>\epsilon\right\}} U_{n,i}^2 \,\,dP_{n,i,{\beta_0}}=0\,.
\end{equation}
Finally, for all $b\in(0,\infty)$ we need
\begin{equation}
\label{cond2}
\lim_{n\rightarrow\infty}\sup_{\left|t\right|\leq b}\sum_{i=1}^{i_{n}}\left\|\sqrt{dP_{n,i,{\beta_0}+t_n}}-\sqrt{dP_{n,i,{\beta_0}}}\big(1+\frac{1}{2}U_{n,i,{\beta_0}}(t)\big)\right\|^{2}_{\mathcal{L}_{2}}=0\,.
\end{equation}
Then, in ${\beta_0}$ and at time $n$, $\mathcal{P}$ has $L_{2}$ derivative $(\Lambda^{\SSs\mathcal{P}}_{n,i,{\beta_0}})$ and
Fisher information  $\mathcal{I}^{\SSs\mathcal{P}}_{n,{\beta_0}}$.\smallskip\\%
$\mathcal{P}$ is continuously differentiable in ${\beta_0}$, if for each sequence $h_n\to0 \in\R^p$,
\begin{equation}
\label{cont2}
\lim_{n\rightarrow\infty}\sup_{\left|t\right|\leq b}\sum_{i=1}^{i_{n}}\left\|\sqrt{dP_{n,i,{\beta_0}+h_n}}
U_{n,i,{\beta_0}+h_n}(t)-\sqrt{dP_{n,i,{\beta_0}}} U_{n,i,{\beta_0}}(t)\right\|^{2}_{\mathcal{L}_{2}}=0\,.
\end{equation}
\end{Def}


Our GLM with deterministic regressors $x_{n,i}\in\R^{p}$ correspondingly is defined as
$\mathcal{P}=\bigotimes_{i=1}^{i_n}\mathcal{P}_{n,i}$
with 
\begin{equation} \label{PdetDef}
\mathcal{P}_{n,i}=
\Big\{\, P_{n,i,{\beta_0}}(dy)=Q_{\vartheta_{n,i}}(dy)\;\Big|\; {\beta_0}\in\R^p;\; \vartheta_{n,i}=\ell(x_{n,i}^{\Tpi} {\beta_0}),\; Q_{\vartheta_{n,i}}\in\mathcal{Q}\,\Big \}\,.
\end{equation}

\citet[Theorem.~2.4.2]{Ri:1994} shows that in
the linear regression case, 
 conditions \eqref{cond1} and \eqref{cond2}
follow from the (uniform) smallness of the hat matrix
$H_n=H_{n;i,j} =  x_{n,i}^\Trsp  (\sum_{g=1}^{i_n} x_{n,g} x_{n,g}^\Trsp )^{-1} x_{n,j}$,
which, as $H_n$ is a projector, reduces to the Feller type condition
\begin{equation} \label{Hmatcond}
\lim_{n} \max_{i=1,\ldots,i_n}  H_{n;i,i} = 0\,.
\end{equation}
In our more general framework, one may still define a corresponding projector $H_n$ locally
(i.e., in ${\beta_0}$) as
\begin{equation}
H_n=H_{n;i,j;{\beta_0}} =  L_{n,i;{\beta_0}}^\Trsp  (\mathcal{I}^{\SSs\mathcal{P}}_{n,{\beta_0}})^{-1}L_{n,j;{\beta_0}},
\quad L_{n,i;{\beta_0}} = \dot\ell(\theta_{n,i})^\Trsp (\mathcal{I}^{\SSs\mathcal{P}}_{n,i,{\beta_0}})^{1/2} \,\dpi\, x_{n,i}
\end{equation}
and, locally, the (changes in the) fitted parameters $\vartheta_{n,i}$ (in a corresponding Fisher scoring
procedure) then can be written as
$$
\vartheta^{\SSs \textrm{(new)}}_{n,i}= \vartheta_{n,i} + \sum_{j=1}^{i_n} (\mathcal{I}^{\SSs\mathcal{P}}_{n,i,{\beta_0}})^{-1/2} H_{n;i,j} (\mathcal{I}^{\SSs\mathcal{P}}_{n,j,{\beta_0}})^{-1/2}
\Lambda^{\SSs\mathcal{Q}}_{\vartheta_{n,j}}(y_{n,j})
.$$ 
However, contrary to the linear regression case, in the general GLM case, the distribution of
the standardized scores $(\mathcal{I}^{\SSs\mathcal{P}}_{n,j,{\beta_0}})^{-1/2}
\Lambda^{\SSs\mathcal{Q}}_{\vartheta_{n,j}}(y_{n,j})$ is not invariant in ${\beta_0}$.
Therefore, the proof for the linear regression fails at this point and condition~\eqref{Hmatcond}
is not sufficient---compare for instance the one-dimensional GLM $\mathcal{P}$ at ${\beta_0}=1$
induced by the one-dimensional Poisson model $\mathcal{Q}$ with parameter $\lambda>0$, $i_n=n$, the
identity as link function and regressors $x_{n,i}=1/n$. In fact, this is the standard
example for a scheme satisfying the Feller condition but violating the Lindeberg condition.
Also, not surprisingly, it is easy to see that Lindeberg condition~\eqref{cond1}
entails condition~\eqref{Hmatcond}.

\begin{Thm}\label{ThmmainDet}
Model $\mathcal{P}$ from \eqref{PdetDef} is continuously $L_{2}$ differentiable
in ${\beta_0}\in\R^p$ with $L_2$ derivative
$\Lambda^{\SSs\mathcal{P}}_{n,i,{\beta_0}}=\Lambda_{\beta_0}^{\SSs\mathcal{P}}(x_{n,i},y)$ with
$\Lambda_{\beta_0}^{\SSs\mathcal{P}}$ from \eqref{LambdaPdef} and Fisher information
$\mathcal{I}^{\SSs\mathcal{P}}_{n,{\beta_0}}$ as given in Definition~\ref{def1}
if the following conditions (i)--(iii) are fulfilled.
\renewcommand{\labelenumi}{(\roman{enumi})}
\begin{enumerate}
\item Model $\mathcal{Q}$ fulfills (H.1)--(H.3). 
\item The  Lindeberg condition~\eqref{cond1} holds for $U_{n,i}$
defined as in Definition~\ref{def1}.
\item  Let $\vartheta_{n,i,t}=\ell(\theta_{n,i,t})$ for
 $\theta_{n,i,t}=x_{n,i}^{\Tpi}\big({\beta_0}+(\mathcal{I}^{\SSs\mathcal{P}}_{n,{\beta_0}})^{-1/2} t\big)$
 and introduce the abbreviations $\mathcal{I}^{\SSs\mathcal{Q}}_{n,i,t}=\mathcal{I}^{\SSs\mathcal{Q}}_{\theta_{n,i,t}}$,
$\dot\ell_{n,i,t}=\dot\ell(\theta_{n,i,t})$, and
$ \mathcal{I}^{\SSs\mathcal{P}}_{n,i,t}=M_\pi\Big(\dot\ell_{n,i,t}^\Trsp\mathcal{I}^{\SSs\mathcal{Q}}_{n,i,t} \dot\ell_{n,i,t},x_{n,i},x_{n,i}\Big)$.
  Then, for every $b\in(0,\infty)$ it must hold
\begin{equation} \label{InfoStet2}
\lim_{n\to\infty}\sup_{|t|\leq b}  \sum_{i=1}^{i_n} \,  t_n^\Trsp (\mathcal{I}^{\SSs\mathcal{P}}_{n,i,t}-\mathcal{I}^{\SSs\mathcal{P}}_{n,i,0}) t_n  = 0\,.
\end{equation}
\end{enumerate}
\end{Thm}

\section{Examples}\label{examples}

\begin{Exam}[{\small Linear regression}]\small\rm
It is obvious that Theorem \ref{Thmmain} can be applied to the linear regression model
\begin{equation}
\label{linmo}
\mathcal{P}=\{P_{\beta}(dx,dy)=F(dy-x^\Trsp \beta)K(dx)\}
\end{equation}
about the one dimensional location model
\begin{equation}
\label{locmo}
\mathcal{Q}=\{Q_{\vartheta}(dy)=F(dy-\vartheta)\}
\end{equation}
for some probability $F$ on $(\R,\B)$ with finite Fisher information
of location\ifx\ArXivS\undefined , if the latter is defined as \fi
$\sup_\varphi (\int \varphi'(x)\,dF)^2/ (\int \varphi^2\,dF)$ where
$\varphi$ varies in the set ${\cal C}^1_0(\R\to\R)$ of all continuously differentiable functions with compact support,
see \citet[Def.~4.1/Thm.~4.2]{Hub:81}---finite Fisher information
of location settles condition~(i) of Theorem~\ref{Thmmain}, condition~(ii)
as already noted boils down to $\int |x|^2\,K(dx)<\infty$ and condition~(iii) is void.


\end{Exam}
\begin{Exam}
[{\small Binomial GLM with logit link and Poisson GLM with log link}]\small\rm\hfill

The Binomial model ${\rm Binom}(m, p)$ for known size $m\in\N$, usually $m=1$, and unknown success
probability $p\in(0,1)$ has error distribution with counting density $q_p(y)={{m}\choose{y}} p^y (1-p)^{m-y}$ (on $y\in\{0,\ldots,m\}$),
hence condition~(i) of Theorem \ref{Thmmain} is obviously
fulfilled with Fisher information $\mathcal{I}^{\SSs\mathcal{Q}}_{p}=m(p(1-p))^{-1}$. Choosing a logit link, i.e.,
$\ell(\theta)=e^{\theta}/(1+e^{\theta})$,
$\mathcal{I}^{\SSs\mathcal{Q}}_{p}\dot\ell(\theta)^2=m p(1-p)$, conditions~(ii) and (iii) become
\ifx\ArXivS\undefined
\begin{eqnarray*}
\mbox{(ii)}&&
\int e^{x^\Trsp \beta}\,(1+e^{x^\Trsp \beta})^{-2}\,|x|^{2}\,K(dx)< \infty,\\
\quad \mbox{(iii)}&&
\int e^{x^\Trsp \beta} \frac{(e^{x^\Trsp s}-1)(1-e^{x^\Trsp (2\beta+s)})}{(1+e^{x^\Trsp (\beta+s)})^{2}(1+e^{x^\Trsp \beta})^{2}} \, |x|^{2}\,K(dx) \rightarrow 0, \quad s\rightarrow 0.
\end{eqnarray*}
\else
\begin{eqnarray*}
\mbox{(ii)}&&
\int \frac{e^{x^\Trsp \beta}}{(1+e^{x^\Trsp \beta})^{2}}|x|^{2}\,K(dx)< \infty,
\quad \mbox{(iii)}\quad 
\int e^{x^\Trsp \beta} \frac{(e^{x^\Trsp s}-1)(1-e^{x^\Trsp (2\beta+s)})}{(1+e^{x^\Trsp (\beta+s)})^{2}(1+e^{x^\Trsp \beta})^{2}} \, |x|^{2}\,K(dx) \rightarrow 0, \quad s\rightarrow 0.
\end{eqnarray*}
\fi
As in these expressions both integrands are bounded pointwise in $x$,
if $|x|^{2}$ is integrable w.r.t.\ $K$, the Binomial GLM with
logit-link is continuously $L_{2}$ differentiable.

The Poisson model ${\rm Pois}(\lambda)$ ($\lambda\in(0,\infty)$) has error distribution with
counting density $q_\lambda(y)=e^{-\lambda} \lambda^y / y!$ (on $y\in\N$), hence
condition~(i) of Theorem \ref{Thmmain} is obviously
fulfilled with Fisher information $\mathcal{I}^{\SSs\mathcal{Q}}_{\lambda}=\lambda^{-1}$. Choosing log link, i.e.,
$\ell(\theta)=e^{\theta}$,
$\mathcal{I}^{\SSs\mathcal{Q}}_{\lambda}\dot\ell(\theta)^2=\lambda$, conditions~(ii) and (iii) become
\begin{eqnarray*}
\mbox{(ii)}\quad 
\int e^{x^\Trsp \beta}|x|^{2} K(dx) < \infty,\qquad
\mbox{(iii)}&&
\int e^{x^\Trsp \beta}(e^{x^\Trsp s}-1) \, |x|^{2}\,K(dx)
\rightarrow 0, \quad s\rightarrow 0.
\end{eqnarray*}
Hence integrability of $e^{|x|(|\beta|+\delta)}|x|^{2}$ w.r.t.\ $K$ implies continuous $L_{2}$ differentiability
of the Poisson GLM with log-link.\medskip

These two conditions, i.e., $|x|\in L_2(K)$ for Binomial logit and $e^{|x|(|\beta|+\delta)}|x|^{2}\in L_1(K)$ for
the Poisson GLM with log-link recover the conditions mentioned in \citet[p.47]{Fa:Tu:01}.
\end{Exam}
\begin{Exam}[{\small GEVD and GPD joint shape-scale models with componentwise log link}]\small\rm\hfill

Both, the generalized extreme value distribution (GEVD) and the generalized Pareto distribution (GPD)
come with a three-dimensional parameter $(\mu,\sigma,\xi)$ for a location or threshold parameter $\mu\in\R$,
a scale parameter $\sigma\in(0,\infty)$ and a shape parameter $\xi\in \R$.
While for the GEVD, in principle the three dimensional model is $L_2$-differentiable for
$\xi\in(-1/2,0)$
and $\xi\in(0,\infty)$, respectively, in the GPD model, the model including the threshold parameter is not
covered by our theory for $L_2$-differentiable error models. The reason is basically, that
observations close to the endpoint of the support in the GPD model carry overwhelmingly much
information on the threshold. To deal with GEVD and GPD in parallel let us hence
assume $\mu$ known in both models, and, for simplicity, $\mu=0$.
Then, parameter $\vartheta$ consists of scale $\sigma$ and shape $\xi$.
In both models,  the scores $\Lambda^{\SSs\mathcal{Q}}_\vartheta$ on the quantile scale,
i.e., $\Lambda_\vartheta(F_\vartheta^{-1}(u))$
for $F_\vartheta^{-1}(u)$ the respective quantile function, include terms of order $(1-u)^\xi$.
Hence for condition~(i), we need to assume that at least $\xi>-1/2$. Depending on the context,
it can be reasonable to add further restrictions. E.g., in case of the GPD, we only obtain an unbounded support
if $\xi\ge 0$; similarly, if we restrict attention
to the special case of Fr\'echet distributions for GEV distributions, $\xi>0$ is a natural
restriction.

For parameter $\vartheta$, we consider a continuously differentiable componentwise link
function $\ell\colon\R^2\to\Theta$, i.e., the link function is of the form
$\ell(\theta)=(\ell_{\sigma}(x_{\sigma}^\Trsp\beta_{\sigma}),\ell_{\xi}(x_{\xi}^\Trsp\beta_{\xi}))$
where we partition the $p$-dimensional regressor $x$ and parameter $\beta$ accordingly to
$x=(x_{\sigma}, x_{\xi})$ and $\beta=(\beta_{\sigma}, \beta_{\xi})$ so that
$\theta=x^{\Tpi} \beta=(x_{\sigma}^\Trsp\beta_{\sigma},x_{\xi}^\Trsp\beta_{\xi})$.
Then, based on the $2\times2$ Fisher information matrix $\mathcal{I}^{\SSs\mathcal{Q}}_{\sigma, \xi}$
for joint scale and shape with entries $I_{\sigma\sigma}$, $I_{\sigma\xi}$ and $I_{\xi\xi}$, we obtain
\begin{equation*}
\dot\ell^\Trsp\mathcal{I}^{\SSs\mathcal{Q}}_{\sigma, \xi}\dot\ell
=\left(
  \begin{array}{cc}
   \dot\ell_{\sigma}^2I_{\sigma\sigma} & \dot\ell_{\sigma}\dot\ell_{\xi}I_{\sigma\xi}\\
    \dot\ell_{\sigma}\dot\ell_{\xi}I_{\sigma\xi} & \dot\ell_{\xi}^2I_{\xi\xi} \\
  \end{array}
\right)\,.
\end{equation*}
That is, conditions~(ii) and (iii) of Theorem~\ref{Thmmain} become
\begin{eqnarray*}
\mbox{(ii)}&&
\int \dot\ell_{\sigma}^2(I_{\sigma\sigma}+I_{\sigma\xi})|x_{\sigma}|^{2}K(dx)+\int\dot\ell_{\xi}^2(I_{\xi\xi}+I_{\sigma\xi})|x_{\xi}|^{2}K(dx)< \infty, \\
\mbox{(iii)}&&
\int (\dot\ell_{\sigma+s}^2(I_{\sigma+s\sigma+s}+I_{\sigma+s\xi+s})-\dot\ell_{\sigma}^2(I_{\sigma\sigma}+I_{\sigma\xi}))|x_{\sigma}|^{2}K(dx)+\\
&&+\int(\dot\ell_{\xi+s}^2(I_{\xi+s\xi+s}+I_{\sigma+s\xi+s})-\dot\ell_{\xi}^2(I_{\xi\xi}+I_{\sigma\xi}))|x_{\xi}|^{2}K(dx) \rightarrow 0, \quad s\rightarrow 0.
\end{eqnarray*}

\noindent\textbf{GEVD model:} The scale-shape model ${\rm GEVD}(0, \sigma, \xi)$ has error distribution $Q_{\vartheta}(y)=\textrm{exp}\big(-(1+\xi\frac{y}{\sigma})^{-\frac{1}{\xi}}\big)$.
As mentioned, condition~(i) of Theorem \ref{Thmmain} is fulfilled as long as $\xi \in (-1/2,0)$
or $\xi>0$. This is reflected by the Fisher information matrix which reads
\begin{equation}
\mathcal{I}^{\SSs\mathcal{Q}}_{\sigma, \xi}=\xi^{-2}D\left(
  \begin{array}{cc}
   I_{\sigma\sigma} & I_{\sigma\xi}\\
    I_{\sigma\xi} & I_{\xi\xi} \\
  \end{array}\label{FIGEV}
\right)D,\quad  \textrm{where}\quad D^{-1}={\rm diag}(\sigma,\xi)\quad \textrm{and}\quad%
\ifx\ArXivS\undefined\vspace*{-4ex}\fi
\end{equation}
\ifx\ArXivS\undefined%
\begin{align*}
I_{\sigma\sigma} &= (\xi+1)^{2}\Gamma(2\xi+1)-2(\xi+1)\Gamma(\xi+1)+1,\\
I_{\sigma\xi}&=-(\xi+1)^{2}\Gamma(2\xi+1)+(\xi^{2}+4\xi+3)\Gamma(\xi+1)+(\xi^{2}+\xi)\Gamma'(\xi)\Gamma(\xi+1)-\xi\Gamma'(1)-\xi-1,\\
I_{\xi\xi}&=(\xi+1)^{2}\Gamma(2\xi+1)-2\Gamma(\xi+3)-2\xi\Gamma'(\xi)\Gamma(\xi+2)+2\xi(\xi+1)\Gamma'(1)+\\&\quad+
\xi^{2}(\Gamma''(1)+(\Gamma'(1))^{2})+(\xi+1)^{2}
\end{align*}
\else%
\begin{align*}
I_{\sigma\sigma} &= (\xi+1)^{2}\Gamma(2\xi+1)-2(\xi+1)\Gamma(\xi+1)+1,\\
I_{\sigma\xi}&=-(\xi+1)^{2}\Gamma(2\xi+1)+(\xi^{2}+4\xi+3)\Gamma(\xi+1)+(\xi^{2}+\xi)\Gamma'(\xi)\Gamma(\xi+1)-\xi\Gamma'(1)-\xi-1,\\
I_{\xi\xi}&=(\xi+1)^{2}\Gamma(2\xi+1)-2\Gamma(\xi+3)-2\xi\Gamma'(\xi)\Gamma(\xi+2)+2\xi(\xi+1)\Gamma'(1)+
\xi^{2}(\Gamma''(1)+(\Gamma'(1))^{2})+(\xi+1)^{2}
\end{align*}
\fi
and has singularities in $\xi=0$ and $\xi=-1/2$.

\noindent\textbf{GPD model:} The scale-shape model ${\rm GPD}(0, \sigma, \xi)$, has a  c.d.f.\ of
$Q_{\vartheta} (y)=1-(1+\xi\frac{y}{\sigma})^{-\frac{1}{\xi}}$ and here,
for $\sigma>0$ and $\xi>-\frac{1}{2}$ condition~(i) is fulfilled with Fisher information matrix:
\begin{equation*}
\mathcal{I}^{\SSs\mathcal{Q}}_{\sigma, \xi}=\frac{1}{1+2\xi}\,D\,\left(\!\!
  \begin{array}{cc}
   1, & 1 \!\!\\
    1, & 2(\xi+1) \!\!
  \end{array}
\right)\,D,\qquad D^{-1}={\rm diag}(\sigma,\xi+1).
\end{equation*}
Again failure of condition~(i) is reflected by  a singularity at $\xi=-1/2$ of the Fisher information.

The canonical link function for the scale is log link
$\ell_{\sigma}(x_{\sigma}^\Trsp\beta_{\sigma})=\exp({x_{\sigma}^\Trsp\beta_{\sigma}})$,
whereas due to a lack of equivariance in the shape, there is no such
canonical link for this parameter. For our  GEVD and GPD applications, however,
(non-regression-based) empirical evidence speaks for shape  $\xi$ varying in $(0, 2)$.
So a good link should not necessarily exclude values $\xi\not\in(0,2)$, but make them
rather hard to attain. For this paper we even impose the sharp restriction $\xi>0$.

Moreover, to use GLMs with GEVD and GPD  errors in
time series context to model parameter driven time dependencies
in the terminology of \citet{Cox:81}, a real challenge is to design
(smooth and isotone) link functions such that the regressors may themselves follow a GEVD or a GPD distribution, as
this implies very heavy tails against which we have to integrate.
More specifically, we aim at constructing an AR-type
time series for the scale and shape of the form
\begin{equation} \label{timeseriesdef}
\hspace*{-.2cm}X_t\sim {\rm GEVD}(\ell(X_{(t-1):(t-p_1)}{\!\!}^\Trsp\beta_{\sigma},X_{(t-1):(t-p_2)}{\!\!}^\Trsp\beta_{\xi}))\quad
\text{for}\;\;X_{(t-1):(t-p)}=(X_{t-1},\ldots,X_{t-p})\,.
\end{equation}
In this model, negative values of $\beta_\xi$ would dampen clustering of
extremes, as then usually a large value stemming from a large positive shape
parameter will be followed by an observation with low or even negative shape
hence with much lighter tails, thus in general a smaller value; correspondingly
$\beta_\xi$ positive will foster clustering of extremes.

A straightforward guess would be to use the log link, but this
does not work for GEVD or GPD time series, as then integrability (ii) fails.
Thus, besides being smooth (for our theorem) and strictly increasing
(for identifyability), an admissible  link function must grow extremely slowly.
To get candidates in case of the GEVD, note that all terms of the Fisher
information matrix for GEVD are dominated by term $\Gamma(2\xi+1)$, so conditions~(ii)
and (iii) are fulfilled if for large positive values $\theta_\xi$, the link function
grows so slowly to $\infty$ that $\Gamma(2\ell_{\xi}(\theta_\xi))\approx \log(\theta_\xi)$,
  which for large $x$ amounts to a behavior like the iterated logarithm $\log(\log(x))$; analogue arguments in case of the GPD suggest  $\ell_{\xi}(\theta_\xi) \approx \log(\theta_\xi)$.

One possibility to achieve this for the GEVD for $p=1$ is $\ell_{\xi}(\theta_\xi)=\log(f(\log(x_\xi)^\Trsp\beta_\xi))$
 where $f(x)$ for $x>0$ is quadratic like $x^2/2+x+1$ and for $x<0$ behaves like $a_1/(\log(a_2-x))^2+a_3$
for some $a_1,a_2,a_3>0$ such that $f$ is continuously differentiable in $0$ and $f(x)>e^{-1/2}$ always.
As is shown in \ref{app4}, this choice indeed fulfills conditions~(ii) and (iii).

\NewInRev{With the singularity in $\xi=0$ of $\mathcal{I}^{\SSs\mathcal{Q}}_{\sigma, \xi}$ in  \eqref{FIGEV},
in many applications, it may turn out useful though to restrict shape~$\xi$ to lie in either $(-1/2,0)$ or in $(0,\infty)$; correspondingly,
one could suggest a rescaled binomial link $\ell=\ell^{\SSs \rm Binom}/2-1/2$ for the first case and shifting the link function $\ell_\xi$
sketched above to $\tilde \ell_\xi=\ell_\xi+1/2$ in the second.\vspace*{-1ex}}
%
\end{Exam}
%
{\small Of course, given an admissible link function, the next question would be
whether for given starting values $x_{-1},\ldots,x_{-\max(p_1,p_2)}$
a time series  defined according to \eqref{timeseriesdef} for $t\ge 0$, using this link function
is (asymptotically) stationary. This is out of scope for this paper and
will be dealt with elsewhere.}

%
\appendix
\ifx\ArXivS\undefined\relax\else\small\fi
\def\theThm{A.\arabic{Thm}}
\def\theHThm{A.\arabic{HThm}}
\section{Proofs}
\NewInRev{\subsection{Proof of H\'ajek's auxiliary result Proposition~\ref{Hajekprop}}
\begin{proof}{}
Apparently, (H.1) and (H.2) are implied by continuous differentiability of the
densities $dQ_\vartheta(y)$ w.r.t.\ $\vartheta$.
\citet{Haj:72} gives a proof of Proposition~\ref{Hajekprop} for $dQ_\vartheta(y)$ Lebesgue densities and for $k=1$, but
our notion of absolute continuity for $k>1$ from p.~\pageref{Hajekprop} reduces the problem to the situation of $k=1$,
which is possible here, as we require (H.1)--(H.3) on open neighborhoods.
In addition, H\'ajek requires (H.1) for every $y$. Looking into his proof of his Lemma~A.2, though,
one does not need that $dQ_\vartheta(y)$ be Lebesgue densities, and in his Lemma~A.3 one only needs
absolute continuity for $Q_{\vartheta_0}$-a.e.\ $y$. Finally, the asserted
continuous $L_2$ differentiability (not mentioned in the cited reference)
with regard to Definition~\ref{def} is just (H.3).
A similar result, already for $k\ge 1$, but only for dominated $\mathcal Q$ and for
continuous differentiability of $dQ_\vartheta(y)$ w.r.t.\ $\vartheta$ for
$Q_{\vartheta_0}$-a.e.\ $y$, is \citet[Satz~1.194]{Wi:1985}.
\end{proof}}


\subsection{Proof of the Chain rule}

\begin{Lem}[Chain rule]\label{Thmchainrule}
Let $\mathcal{Q}=\{Q_\vartheta\,\mid\,\vartheta\in\Theta\}$ a parametric model with open parameter domain $\Theta\subset\R^k$.
Assume $\mathcal{Q}$ is $L_{2}$ differentiable in $\vartheta_{0}\in \Theta$ with derivative $\Lambda_{\vartheta_{0}}^{\SSs\mathcal{Q}}$ and Fisher information $I_{\vartheta_{0}}^{\SSs\mathcal{Q}}$. Let $\ell\colon \Theta'\to \Theta$ with domain $\Theta'\subset \R^{k'}$ be differentiable in some $\theta_0\in\Theta'$ such that $\ell(\theta_0)=\vartheta_0$ and with derivative denoted by $\dot\ell(\theta_0)$.
Then $\tilde{\mathcal{Q}}=\{\tilde{Q}_{\vartheta}=Q_{\ell(\theta)}\,\mid\,\theta\in\Theta'\}$
is $L_{2}$ differentiable in $\theta_{0}$ with derivative
$\Lambda_{\theta}^{\SSs\tilde{\mathcal{Q}}}=(\dot{\ell}(\theta_{0}))^\Trsp \Lambda_{\vartheta_{0}}^{\SSs\mathcal{Q}}$
and Fisher information $I_{\theta}^{\SSs\tilde{\mathcal{Q}}}=(\dot{l}(\theta_{0}))^\Trsp I_{\vartheta_{0}}^{\SSs\mathcal{Q}}\dot{\ell}(\theta_{0})$.
If $\mathcal{Q}$ is continuously $L_{2}$ differentiable in $\vartheta_{0}$, so is $\tilde {\mathcal{Q}}$ in $\theta_0$.
\end{Lem}

\begin{proof}{}
Let $h_{n}\rightarrow 0, n\rightarrow\infty$ in $\R^{k'}$, $|h_{n}|\neq 0$.
We take $\vartheta_{n}:=\ell(\theta_{0}+h_{n})$, $\vartheta_{0}:=\ell(\vartheta_{0})$.
Smoothness of link function $\ell$ implies:
\begin{equation}
\label{smoo}
\vartheta_{n}=\ell(\theta_{0}+h_{n})=\vartheta_{0}+\dot{\ell}(\theta_{0})h_{n}+r(\theta_{0},h_{n}),
\end{equation}
for some remainder function $r$ such that
\begin{equation}
\label{rest1}
\lim_{n\to \infty}r(\theta_{0},h_{n})/|h_{n}| =0\,.
\end{equation}
Let $Q_{\vartheta_{n}}$ be dominated by some measure $\nu$ with 
density $q_{\vartheta_{n}}$, i.e.,
$
dQ_{\vartheta_{n}}=q_{\vartheta_{n}}d\nu
$. 
By $L_{2}$ differentiability of model $Q$ for
$R_{n}:=\int \Big(\sqrt{q_{\vartheta_{n}}}-\sqrt{q_{\vartheta_{0}}}\big(1+\frac{1}{2}(\Lambda_{\vartheta_{0}}^{\SSs\mathcal{Q}})^\Trsp (\vartheta_{n}-\vartheta_{0})\big)\Big)^{2}d\nu$,
we have
\begin{equation}
\label{rest2}
\lim_{n\to\infty} R_{n}/|\vartheta_{n}-\vartheta_{0}|^{2}=0\,.
\end{equation}
But by (\ref{smoo}) we may write $R_n$ as $R_{n}=\int(A_{n}-B_{n})^{2}\,d\nu $ for
\begin{equation*}
A_{n}:=\sqrt{q_{\vartheta_{n}}}-\sqrt{q_{\vartheta_{0}}}\Big(1+\frac{1}{2}(\Lambda_{\vartheta_{0}}^{\SSs\mathcal{Q}})^\Trsp \dot{l}(\vartheta_{0})h_{n}\Big)\quad \textrm{and} \quad
B_{n}:=\frac{1}{2}\sqrt{q_{\vartheta_{0}}}(\Lambda_{\vartheta_{0}}^{\SSs\mathcal{Q}})^\Trsp r(\vartheta_{0},h_{n}).
\end{equation*}
Now, Cauchy-Schwarz entails that
$
A_{n}^{2}
\leq 2(A_{n}-B_{n})^{2}+2B_{n}^{2}.
$ 
Therefore
\begin{eqnarray*}
\int A_{n}^{2}d\nu&\leq& 2\int(A_{n}-B_{n})^{2}d\nu+2\int B_{n}^{2}d\nu=
2R_{n}+2\int B_{n}^{2}d\nu\leq\\
&\leq&2R_{n}+\frac{1}{2}|r(\vartheta_{0},h_{n})|^{2}\int q_{\vartheta_{0}}|\Lambda_{\vartheta_{0}}^{\SSs\mathcal{Q}}|^{2}d\nu\leq
2R_{n}+\frac{1}{2}|I_{\vartheta_{0}}^{\SSs\mathcal{Q}}||r(\vartheta_{0},h_{n})|^{2}.
\end{eqnarray*}
Hence, using \eqref{smoo}, \eqref{rest1}, and \eqref{rest2}
$$
\frac{1}{|h_{n}|^{2}}\int A_{n}^{2}d\nu = 
\frac{2R_{n}}{|\vartheta_{n}-\vartheta_{0}|^{2}}\frac{\Big(\dot{l}(\vartheta_{0})h_{n}+r(\vartheta_{0},h_{n})\Big)^2}{|h_{n}|^{2}}+\frac{1}{2}|I_{\vartheta_{0}}^{\SSs\mathcal{Q}}|\frac{|r(\vartheta_{0},h_{n})|^{2}}{|h_{n}|^{2}}=
\Lo(1). 
$$
That is, by Definition \ref{def} model $\tilde{Q}$
is $L_{2}$ differentiable in $\vartheta_{0}\in \Theta'$.
\end{proof}


\subsection{Proof of Theorem~\ref{Thmmain}}

Let $s_n\rightarrow 0 \in\R^p$ for $n\to\infty$ such that $\tilde s_n=s_n/|s_n| \to \tilde s_0$
for some $\tilde s_0$ with $|\tilde s_0|=1$.
We take $\vartheta_{s}:=\ell(\theta_{s})$, $\theta_{s}:=x^\Trsp ({\beta_0}+s)$,
$\dot\ell_s=\dot\ell(\theta_{s})$.
Let $dQ_{\vartheta_{n}}=q_{\vartheta_{n}}\,d\nu$.
By Definition~\ref{def} the GLM $\mathcal{P}$
is $L_{2}$ differentiable at every $\beta\in\R^p$ if
$
\lim\limits_{n\to\infty}|s_{n}|^{-2}\!\!\int\int\tilde{A}_{n}^{2}\,\nu(dy)\,K(dx)=0
$ 
for the $A_n$ from Lemma~\ref{Thmchainrule} now taking up the dependence on $x$, i.e.,
\begin{equation} \label{TAnDef}
\tilde{A}_{n}=\tilde{A}_{n}(x,y):=\sqrt{q_{\vartheta_{n}}}-\sqrt{q_{\vartheta_{0}}}\,\Big(1+\frac{1}{2}(\Lambda_{\ell(x^\Trsp {\beta_0})}^{\SSs\mathcal{Q}})^\Trsp \dot{\ell}(x^\Trsp {\beta_0}) \dpi x^\Trsp s_{n}\Big)\,.
\end{equation}
Here (pointwise) existence (for $P_\beta$-a.e.\ $(x,y)$) and form of the $L_2$-derivative
follow from (H.1) and the chain rule applied pointwise (in $(x,y)$).
%
The proof of Lemma~\ref{Thmchainrule} for $K$-a.e.\ $x$ and $s$ small enough
provides some function $z(s)\rightarrow0$ such that
\begin{equation*}
\int\tilde{A}_{n}^{2}\nu(dy)=|x^\Trsp s_{n}|^{2}(z(x^\Trsp s_{n}))^{2}.
\end{equation*}
Hence, for $K$-a.e.\ fixed $x$, $\tilde A_n'(x):=|s_n|^{-2} \int\tilde{A}_{n}^{2}\nu(dy) \to 0$.
For Lebesgue measure $\lambda$, fixed $x\in\R^p$ and $u\in[0,1]$ by the fundamental theorem of calculus
for absolutely continuous functions, for $K$-a.e.\ fixed $x$ we obtain
\ifx\ArXivS\undefined%
\begin{eqnarray*}
\hspace*{-0.5cm}&\!\!\!\!&|s_n|^{-2}\!\!\int\Big(\sqrt{q_{\vartheta_{s_n}}}-\sqrt{q_{\vartheta_{0}}}\Big)^{2}\,d\nu=|s_n|^{-2}\!\!\int\Big(\int^{1}_{0}
\Tfrac{1}{2}\sqrt{q_{\vartheta_{us_n}}}(\dot{\ell}_{us_n}^\Trsp\Lambda_{\vartheta_{us_n}}^{\SSs\mathcal{Q}} \!\!\dpi x^\Trsp s_n)\,\lambda(du)
\Big)^{2}\,d\nu\leq\\
\hspace*{-0.5cm}&\!\!\!\!&\quad\leq\Tfrac{1}{4|s_n|^{2}}\int\,\int^{1}_{0}q_{\vartheta_{us_n}}(\dot{\ell}_{us_n}^ \Trsp\Lambda_{\vartheta_{us_n}}^{\SSs\mathcal{Q}} \dpi x^\Trsp s_n)^{2}\,\lambda(du)\,d\nu
= \frac{1}{4} \tilde s_n^\Trsp \int^{1}_{0}  \mathcal{I}_{\vartheta_{us_n}}^{\SSs\mathcal{P}}(x) \,\lambda(du)\,\tilde s_n=
\\ \hspace*{-0.5cm}&\!\!\!\!&\quad=
\frac{1}{4|s_n|} \tilde s_n^\Trsp \int^{|s_n|}_{0} \mathcal{I}_{\vartheta_{u\tilde s_n}}^{\SSs\mathcal{P}}(x)
\,\lambda(du)\,\tilde s_n
=: B_n(x)\,.
\end{eqnarray*}
\else
\begin{eqnarray*}
\hspace*{-0.5cm}&\!\!\!\!&|s_n|^{-2}\!\!\int\Big(\sqrt{q_{\vartheta_{s_n}}}-\sqrt{q_{\vartheta_{0}}}\Big)^{2}\,d\nu=|s_n|^{-2}\!\!\int\Big(\int^{1}_{0}
\Tfrac{1}{2}\sqrt{q_{\vartheta_{us_n}}}(\dot{\ell}_{us_n}^\Trsp\Lambda_{\vartheta_{us_n}}^{\SSs\mathcal{Q}} \!\!\dpi x^\Trsp s_n)\,\lambda(du)
\Big)^{2}\,d\nu\leq\\
\hspace*{-0.5cm}&\!\!\!\!&\quad\leq\Tfrac{1}{4|s_n|^{2}}\int\,\int^{1}_{0}q_{\vartheta_{us_n}}(\dot{\ell}_{us_n}^ \Trsp\Lambda_{\vartheta_{us_n}}^{\SSs\mathcal{Q}} \dpi x^\Trsp s_n)^{2}\,\lambda(du)\,d\nu
= \frac{1}{4} \tilde s_n^\Trsp \int^{1}_{0}  \mathcal{I}_{\vartheta_{us_n}}^{\SSs\mathcal{P}}(x) \,\lambda(du)\,\tilde s_n=
\frac{1}{4|s_n|} \tilde s_n^\Trsp \int^{|s_n|}_{0} \mathcal{I}_{\vartheta_{u\tilde s_n}}^{\SSs\mathcal{P}}(x)
\,\lambda(du)\,\tilde s_n
=: B_n(x)\,.
\end{eqnarray*}
\fi
Now, introduce $B_0=\tilde s_n^\Trsp \mathcal{I}_{\vartheta_{0}}^{\SSs\mathcal{P}}(x) \tilde s_n / 4$.
By (ii) and (iii) $\int B_n(x) \,K(dx)$ is finite eventually in $n$, and by (iii) and Fubini
$$\int B_n(x) \,K(dx) = \frac{1}{4} \int^{|s_n|}_{0} \int |\mathcal{I}_{\vartheta_{u\tilde s_n}}^{\SSs\mathcal{P}}(x)| \,K(dx)\, \lambda(du) = \int B_0(x) \,K(dx) + \Lo(1)\,.
$$
Hence, by Vitali's Theorem (e.g.\ \citet[Prop.~A.2.2]{Ri:1994}),
$B_n$ is uniformly integrable (w.r.t.\ $K$), and, as $\tilde A_n'(x)\leq 2B_n(x)+2B_0(x)$,
so is $\tilde A_n'(x)$, and again by Vitali's Theorem, $\int \tilde A_n'(x)\, K(dx)\to 0$ which is \eqref{L2dk}.
Continuity \eqref{L2dc} with regard to Vitali's Theorem is just continuity of the Fisher information just shown.\smallskip\\
The assertion of Remark~\ref{Rem26} is shown similarly, replacing the $B_n$ and $B_0$ from above
with $|\mathcal{I}^{\SSs\mathcal{Q}}_{\vartheta_{st}}| \,|\dot{\ell}_{st}|^2 \,|x|^2$ resp.\
$|\mathcal{I}^{\SSs\mathcal{Q}}_{\vartheta_{0}}| \,|\dot{\ell}_{0}|^2 \,|x|^2$.
\qed

\subsection{Proof of Theorem~\ref{ThmmainDet}}

For selfcontainedness, we reproduce the argument for condition~\eqref{cent} from \citet[Thm.~2.3.7]{Ri:1994}.
In model $\mathcal{Q}$, by \eqref{L2dk}, assuming $\nu$-densities
\ifx\ArXivS\undefined
\begin{eqnarray*}
\Big|\!\int \!\Big(\!\sqrt{q_{\vartheta+h}}-\sqrt{q_{\vartheta}}(1+\Tfrac{1}{2}(\Lambda^{\SSs\mathcal{Q}}_{\vartheta})^\Trsp h)\!\Big) \sqrt{q_{\vartheta}}\,d\nu\,\Big|^2\leq 
\int |\sqrt{q_{\vartheta+h}}-\sqrt{q_{\vartheta}}(1+\Tfrac{1}{2}(\Lambda^{\SSs\mathcal{Q}}_{\vartheta})^\Trsp h)|^2\,d\nu
\end{eqnarray*}
where the RHS by $L_2$-differentiability of ${\cal Q}$ is $\Lo(|h|^2)$.
\else
\begin{eqnarray*}
&&\Big|\int \Big(\sqrt{q_{\vartheta+h}}-\sqrt{q_{\vartheta}}(1+\frac{1}{2}(\Lambda^{\SSs\mathcal{Q}}_{\vartheta})^\Trsp h)\Big) \sqrt{q_{\vartheta}}\,d\nu\,\Big|^2\leq 
\int |\sqrt{q_{\vartheta+h}}-\sqrt{q_{\vartheta}}(1+\frac{1}{2}(\Lambda^{\SSs\mathcal{Q}}_{\vartheta})^\Trsp h)|^2\,d\nu = \Lo(|h|^2)
\end{eqnarray*}
\fi
Hence,
\ifx\ArXivS\undefined
\begin{eqnarray*}
\Ew_\vartheta (\Lambda^{\SSs\mathcal{Q}}_{\vartheta})^\Trsp h + \Lo(|h|)\!\!\!&\!\!\!=\!\!\!&\!\!\!
\int (\sqrt{q_{\vartheta+h}}-\sqrt{q_{\vartheta}}) \sqrt{q_{\vartheta}} \,d\nu =
- \int (\sqrt{q_{\vartheta+h}}-\sqrt{q_{\vartheta}})^2 \,d\nu / 2 =\\
\!\!\!&\!\!\!=\!\!\!&\!\!\!
- h^\Trsp \mathcal{I}^{\SSs\mathcal{Q}}_{\vartheta} h/2 +\Lo(|h|^2)=
\Lo(|h|)\,.
\end{eqnarray*}
\else
\begin{eqnarray*}
\Ew_\vartheta (\Lambda^{\SSs\mathcal{Q}}_{\vartheta})^\Trsp h\!\!\!&\!\!\!=\!\!\!&\!\!\!
\int (\sqrt{q_{\vartheta+h}}-\sqrt{q_{\vartheta}}) \sqrt{q_{\vartheta}} \,d\nu + \Lo(|h|)=
\int \sqrt{q_{\vartheta+h}} \sqrt{q_{\vartheta}} \,d\nu -1 + \Lo(|h|)=\\
\!\!\!&\!\!\!=\!\!\!&\!\!\!
-1/2 \int (\sqrt{q_{\vartheta+h}}-\sqrt{q_{\vartheta}})^2 \,d\nu + \Lo(|h|)=-1/2 h^\Trsp \mathcal{I}^{\SSs\mathcal{Q}}_{\vartheta} h +\Lo(|h|^2)+\Lo(|h|)=\Lo(|h|)
\end{eqnarray*}
\fi
So $\Ew_\vartheta \Lambda^{\SSs\mathcal{Q}}_{\vartheta}=0$, and hence also
$\Ew_{n,i,{\beta_0}} \Lambda^{\SSs\mathcal{P}}_{n,i,{\beta_0}}=0$. Lindeberg condition~\eqref{cond1} is assumed
without change, so it only remains to show condition~\eqref{cond2}.
Let $N_{n,i}$ be the $Q_{\vartheta_{n,i,t_n}}$-null set such that both (H.1) and (H.2) hold for
all $y\in N_{n,i}^{c}$. Let $N=\bigcup_n\bigcup_{i=1}^{i_n} N_{n,i}$.  Then as in the case of
stochastic regressors, from (H.1) and the chain rule
applied pointwise (in $y\in N^c$) we obtain (pointwise)
existence and form of the $L_2$-derivative.
Let $\tilde A_n$ from \eqref{TAnDef} now take up the dependence on $x_{n,i}$, i.e.,
$
\tilde{A}_{n,i}=\tilde{A}_{n}(x_{n,i})
$  (with $s_n$ from the preceding proof substituted by $t_n$)
so that in particular, for every fixed $i$,
$\tilde A_{n,i}':=\int\tilde{A}_{n,i}^{2}\nu(dy) \to 0$ as $t_n\to 0$.
For condition~\eqref{cond2} we have to show that
$
\lim_{n\rightarrow\infty}\sup_{\left|t\right|\leq b} \sum_{i=1}^{i_{n}} \int \, \tilde A_{n,i}^2 \nu(dy)=0
$.
But, similarly as in the preceding proof for fixed $i$, 
by the fundamental theorem of calculus for absolutely continuous functions, we have
\begin{eqnarray*}
\tilde A_{n,i}'&=&\int\Big(\sqrt{q_{\vartheta_{n,i,t_n}}}-\sqrt{q_{\vartheta_{n,i,0}}}\Big)^{2}\,d\nu
\leq \frac{1}{4|t_n|} \int^{|t_n|}_{0} t_n^\Trsp \mathcal{I}^{\SSs\mathcal{P}}_{n,i,ut} t_n \,\lambda(du) =: B_{n,i}\,.
\end{eqnarray*}
Now, introduce $B_{0,i}=\frac{1}{4} t_n^\Trsp \mathcal{I}^{\SSs\mathcal{P}}_{n,i,0} t_n$ and
note that $\sum_{i=1}^{i_n} \mathcal{I}^{\SSs\mathcal{P}}_{n,i,0} = \mathcal{I}^{\SSs\mathcal{P}}_{n,{\beta_0}}$,
so $t_n^\Trsp \mathcal{I}^{\SSs\mathcal{P}}_{n,i,0} t_n= |t|^2 \leq b$ and by (iii)
$\sum_{i=1}^{i_n}  B_{n,i} = \sum_{i=1}^{i_n}  B_{0,i} + \Lo(1) = |t|^2/4 + \Lo(1)
$.
Hence, by Vitali's Theorem, $B_{n,i}$ is uniformly integrable (w.r.t. the counting measure), and, as $\tilde A_{n,i}'\leq 2B_{n,i}+2B_{0,i}$,
so is $\tilde A_{n,i}'$, and again by Vitali's Theorem, $\sum_{i=1}^{i_n} \tilde A_{n,i}' \to 0$.
Finally, continuity~\eqref{cont2} again with regard to Vitali's Theorem  is just continuity of the Fisher information
just proven.\smallskip\\
\qed


\subsection{Link function for the GEVD shape model}
\label{app4}

For GEVD for the shape we have chosen link function $\ell=\log(f(\beta\log(x_{t-1}))$, for $$f(x)=(x^2/2+x+1) \Jc(x>0)+(a_1(\log(a_2-x))^{-2}+a_3) \Jc(x\le 0)$$ for some $a_1,a_2,a_3>0$. The constants $a_1,a_2,a_3$ are chosen so that $f$ is continuously differentiable in $0$ and $f(x)>e^{-1/2}$ always, i.e.
\begin{equation} \label{linkvors}
\frac{a_1}{(\log(a_2))^2}+a_3=\frac{2a_1}{a_2(\log(a_2))^3}=1, \quad \frac{a_1}{(\log(a_2-x))^2}+a_3>e^{-1/2}, \forall x<0.
\end{equation}
Since $a_1(\log(a_2-x))^{-2}>0$, to ensure the last inequality we let $a_3=e^{-1/2}\approx0.6063$. Solving the system of equations we get  $a_2^{a_2}=e^{2(1-e^{-0.5})}$, so $ a_2\approx1.624$ and $a_1=0.5a_2(\log(a_2))^3\approx0.00926$.

As said, shape is usually varying in $(0,2)$. As visible from the Figure~\ref{Rplot}, this interval corresponds to  an argument
of the link function $x=\beta\log(x_{t-1})$ ranging in $(-\infty, \sqrt{1-2(1-e^2)}-1\approx 2.712)$;
hence, for $\beta=1$, $\ell=\log(f(\beta \log(x_{t-1})))$ is smaller than $2$ as long as $x_{t-1}<15$ and $\ell<3$ for $x_{t-1}<193$.

\begin{figure}
\centering
\includegraphics[width=0.7\linewidth]{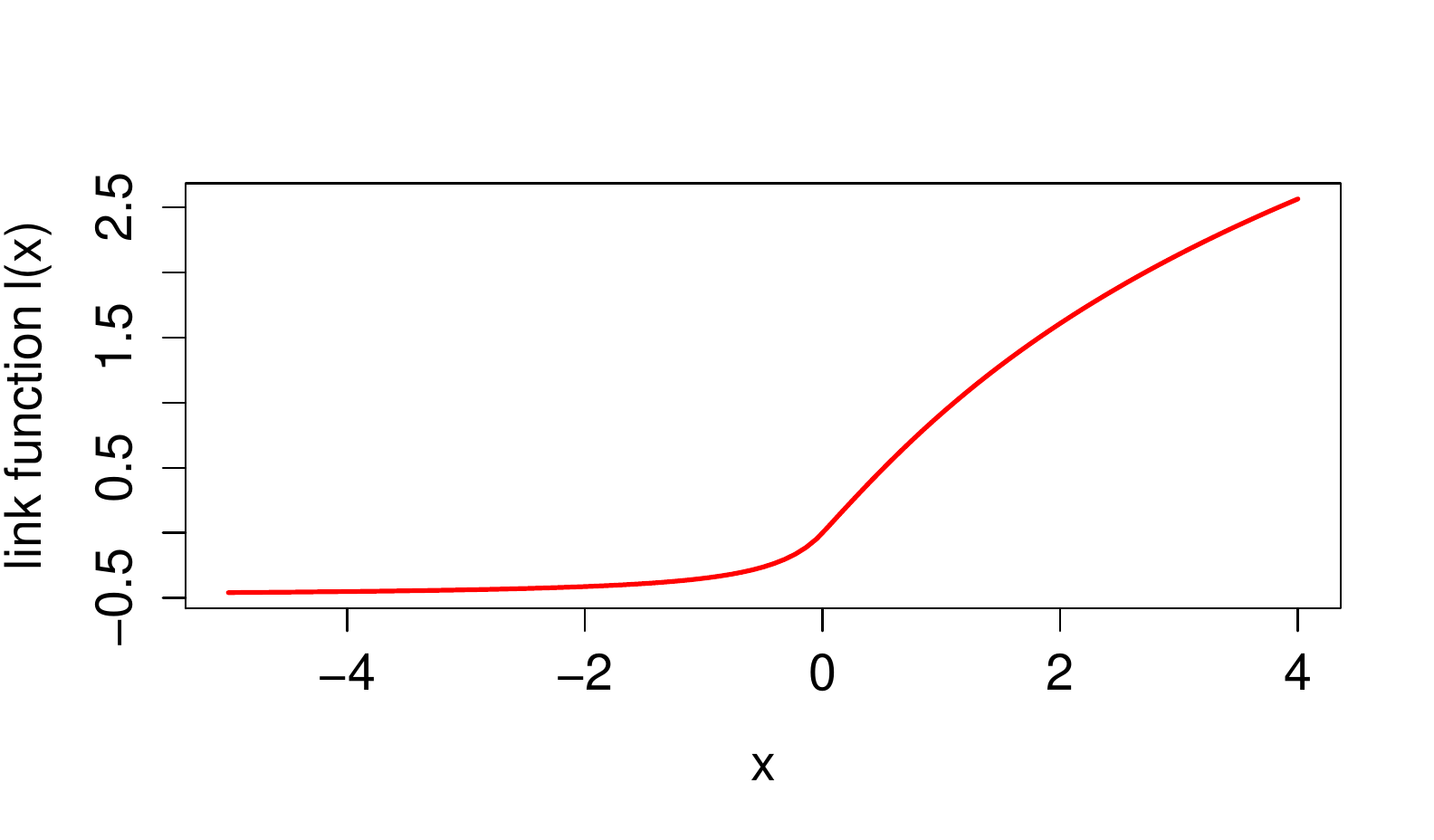}\vspace*{-3ex}
\caption{Link function for the shape of GEVD}
\label{Rplot}
\end{figure}

To show that our choice of link function for GEVD,
fulfills conditions~(ii) and (iii), first we calculate its derivative $\dot{\ell}=\dot{f}/f$ and obtain
$\dot{\ell}=(x+1)/(x^2/2+x+1)$ for $x>0$ and  $\dot{\ell}=2a_1(a_2-x)^{-1}(\log(a_2-x))^{-3}$ for $x<0$.
Hence, for large $x$, $\dot \ell$ behaves like $2/x$, while
for $x<0$, it essentially behaves like $-x^{-1}(\log(-x))^{-3}$.

As mentioned, the term $\Gamma(2x)$ dominates all entries of all terms of $\mathcal{I}^{\SSs\mathcal{Q}}_{\sigma, \xi}$
in \eqref{FIGEV}. Using
the Stirling approximation, i.e., $\Gamma(x)\approx \sqrt{2\pi}\,\exp(x(\log(x)-1/2))$,
due to the double application of the logarithm in the link function we get that $\Gamma(2\ell_{\xi}(\theta_\xi))$ is approximately  $\beta_\xi\log(x_\xi)$.
By equivariance in $\mu$ and $\sigma$, therefore the integral of condition~(ii) turns into:
%
$B_{1}(\xi):=
4\beta_\xi^{-1}\int\log(x)\,K(dx) <\infty$
for $\beta_\xi>0$
 and, for $\beta_\xi<0$, to
$B_{2}(\xi):
=\beta_\xi^{-1}\int \,\log(x) \big((\log(-\beta_\xi)+\log(\log(x))\big)^{-6}\,K(dx)$.
%
Finiteness of 
$B_{1}(\xi)$ and $B_{2}(\xi)$ follow from finiteness of $\Ew(\min\{1, (\log x)^k\})$ for $x\sim \textrm{GEVD}(0,1,\xi)$, $k\in \textbf{N}$.
%
%
Reconsidering $B_{1}(\xi)$, $B_{2}(\xi)$ at $\xi+s$, for $\left|s\right|<h$, $h<1$ we see that
$\sup_{\left|s\right|<h}B_{i}(\xi+s)<\infty$ for $i=1,2$, hence, condition~(iii) is a consequence
of dominated convergence and continuity of Fisher information $I_{\xi\xi}$ in $\xi$.

\section*{Acknowledgement}
This article is part of the PhD thesis of Daria Pupashenko.
All authors gratefully acknowledge financial support by the Volkswagen Foundation
for the project ``Robust Risk Estimation'', 
{\footnotesize \url{http://www.mathematik.uni-kl.de/~wwwfm/RobustRiskEstimation}}.
\ifx\ArXivS\undefined
\NewInRev{We also thank the associate editor for his helpful comments
and drawing our attention to applications with $\alpha$-stable error
distributions.}\fi
%
\ifx\ArXivS\undefined
\ifx\StProb\undefined
\section*{References}
\fi
\fi

%
%
\end{document}